\documentclass[a4paper,12pt]{amsart}
\usepackage[T1]{fontenc}
\usepackage[centering,margin=2.3cm]{geometry}
\usepackage[latin1]{inputenc}
\usepackage{amsthm, amsmath}   
\usepackage{latexsym,amssymb}
\usepackage{amssymb}

\usepackage{amsaddr}
\usepackage{graphicx,subfigure}
\usepackage{times}
\usepackage{enumerate,url}
\usepackage{epsfig}
\usepackage{color}
\usepackage{cite}
\newcommand{\Cay}{\mathrm{Cay}}

\theoremstyle{plain}
\newtheorem{theorem}{Theorem}[section]
\newtheorem{lemma}[theorem]{Lemma}

\newtheorem{proposition}[theorem]{Proposition}

\theoremstyle{definition}

\newtheorem{conjecture}[theorem]{Conjecture}

\newtheorem{question}[theorem]{Question}
\newtheorem{remark}[theorem]{Remark}

\title[Coloring minimal Cayley graphs]{Coloring minimal Cayley graphs}
\author[I. Garc\'{i}a-Marco]{Ignacio Garc\'{i}a-Marco}
\address{Facultad de Ciencias and Instituto de Matematicas y Aplicaciones (IMAULL), \\ Universidad de La Laguna, Spain}
\author[K. Knauer]{Kolja Knauer$^*$}
\address{Aix Marseille Univ, Universit\'e de Toulon, CNRS, LIS, Marseille, France\\Departament de Matem\`atiques i Inform\`atica,
Universitat de Barcelona, Spain}

\keywords{Chromatic number, minimal Cayley graph, nilpotent group \\ \ \ \ $ ^*$ Corresponding author}


\begin{document}

\begin{abstract}
In 1978 Babai raised the question whether all minimal Cayley graphs have bounded chromatic number; in 1994 he conjectured a negative answer. In this paper we show that any minimal Cayley graph of a (finitely generated) generalized dihedral or nilpotent group has chromatic number at most 3, while 4 colors are sometimes necessary for soluble groups. On the other hand we address a related question proposed by Babai in 1978 by constructing graphs of unbounded chromatic number that admit a proper edge coloring such that each cycle has some color at least twice. The latter can be viewed as a step towards confirming Babai's 1994 conjecture -- a problem that remains open.
\end{abstract}

\maketitle

\section{Introduction}
Given a group $\Gamma$ and a \emph{connection set} $C\subseteq\Gamma$ the (undirected, right) \emph{Cayley graph} $\Cay(\Gamma, C)$ has vertex set $\Gamma$ and $a,b\in\Gamma$ form an edge if $a^{-1}b\in C$. Note that we will consider Cayley graphs as undirected even if the connection set is not inverse-closed. A Cayley graph is \emph{minimal} (a.k.a irreducible) if $C$ is an inclusion-minimal generating set of $\Gamma$ and in this paper we only consider the case where $C$ is finite. Minimal Cayley graphs appear naturally: a famous and open problem often attributed to Lov\'asz is (equivalent to) whether minimal Cayley graphs are Hamiltonian, see~\cite[Section 4]{PR09}. Other areas in which minimal Cayley graphs naturally occur are the genus of a group, see the book~\cite{Whi01} or concerning sensitivity~\cite[Questions 7.7 and 8.2]{GMK22}. 
Different aspects of these graphs (sometimes of particular groups) and their distinguishing features compared to general Cayley graphs have been considered in~\cite{MS20,SS09,LZ01,HI96,God81,FM16}. The present work is motivated by a question that has been brought up by Babai~\cite{B78,B95}:

\begin{question}\label{question}
 Does there exist a finite constant $c$ such that every minimal Cayley graph has chromatic number at most $c$?
\end{question}

Babai conjectured in 1994 a negative answer even in the case the group is finite. 
Note that assuming minimality here is essential. Indeed, every graph is an induced subgraph of some Cayley graph~\cite{BS85,GI87,Bab78}, this is far from being the case for minimal Cayley graphs, which are known to be sparse~\cite{B78,Spe83}. However, there also exist (non-minimal) sparse Cayley graphs (of arbitrary girth) that are expanders and hence have unbounded chromatic number~\cite{Lub-88}. The chromatic number of  random Cayley graphs has been studied by Alon~\cite{Alo13} and Green~\cite{Gre17} and some concrete Cayley graphs have been considered recently~\cite{CK23,Dav23}. However, none of these are minimal.
Concerning the chromatic number of minimal Cayley graphs, in~\cite[Section 3.4]{B95} Babai mentions the following:

\begin{conjecture}\label{conj:babai_new} For every $\varepsilon>0$ there exists a minimal Cayley graph $G=(V,E)$ such that $\alpha(G) \leq \varepsilon|V|$  where $\alpha(G)$
denotes the size of the largest independent set of $G$. 
\end{conjecture}

Hence, this in particular would imply a negative answer to Question~\ref{question}.
However, previously in~\cite{B78} Babai had proposed an opposing conjecture (which he withdrew in 1994 when he proposed Conjecture~\ref{conj:babai_new}), for which we introduce some notions first. 
A graph $G$ is called \emph{no lonely color} if $G$ admits an edge coloring satisfying:
\begin{enumerate}
\item each vertex is incident with at most two edges of any color,
\item each circuit contains no color exactly once.
\end{enumerate}
In a Cayley graph the edges can be naturally colored by the elements of $C$, which makes it easy to see that any subgraph $G$ of a minimal Cayley graph is \emph{no lonely color}.
A Cayley graph $\Cay(\Gamma,C)$ is \emph{semiminimal} if $C$ is a generating set of $\Gamma$ that can be linearly ordered such that none of its elements is generated by its predecessors. With the same ideas one can show that any subgraph $G$ of a semiminimal Cayley graph is \emph{one popular color}\footnote{Babai called these \emph{no pied circuit}.}, this is, $G$ admits an edge coloring that satisfies (1) and \begin{itemize}
                            
                        \item[(2')] each circuit contains at least one color more than once.
                                                                                                                                                                                                                                                                                                                                                           \end{itemize}

Babai observes that $K_4-e$ and $K_{3,5}$ are not no lonely color and $K_{5,17}$ is not one popular color, hence they cannot be subgraphs of minimal and semiminimal Cayley graphs, respectively. Later using these ideas Spencer~\cite{Spe83} shows that for every $g \geq 3$  there exists a finite graph $G$ of girth $g$ that is not a subgraph of any (semi)minimal Cayley graph.  In~\cite[Conjecture 3.5]{B78} Babai puts forward the following:
\begin{question}\label{conj:babai_old}
Is there a finite constant $c$ such that any one popular color graph (no lonely color graph) has chromatic number at most $c$? 
\end{question}
 Note that a positive answer to Question~\ref{conj:babai_old} would imply that (semi)minimal Cayley graphs have bounded chromatic number, and hence would disprove the more recent Conjecture~\ref{conj:babai_new} and give a positive answer to Question~\ref{question}.

%

 This is the starting point for the present work: 
 First, we observe that every group admits a minimal Cayley graph of chromatic number at most $3$ and characterize the bipartite case (Theorem \ref{th:babai}). 
 Then we show that minimal Cayley graphs of finitely generated nilpotent groups (Theorem \ref{thm:abelian3}) as well as generalized dihedral groups (Theorem \ref{thm:gendihedral3}) have chromatic number at most $3$. On the other hand, we show a minimal Cayley graph of a soluble group of chromatic number $4$ and a semiminimal Cayley graph of a nilpotent group of chromatic number $7$ (see Proposition~\ref{prop:lb}). 
 Finally, we negatively answer the stronger statement of Question~\ref{conj:babai_old} by constructing a family of one popular color graphs of unbounded chromatic number (Theorem \ref{th:countbabai}). We further remark, that the clique number of a semiminimal Cayley graph is at most $4$ (Proposition \ref{prop:clique}).
 
                                                                                                                                                                                                                                                \section{Positive results}

Let us begin with the question whether every finite group admits a minimal Cayley graph of bounded chromatic number. This was already answered by Babai~\cite{B78,B95}. Indeed, he proved that every finite group has a 3-colorable minimal Cayley graph.  His proof depends on the fact that every finite simple group can be generated by two elements, one of which is an involution.  This is a known consequence of the Classification of Finite Simple Groups.  For completeness we include a proof following Babai's lines. At the basis of this we have the following~\cite[Lemma 4.2]{B78}, which we will also use later on. For a group $\Gamma$ denote its \emph{minimum chromatic number} and $\chi_{\mathrm{min}}(\Gamma)$ the minimum $\chi(\Cay(\Gamma, C))$ over all generating sets $C$ of $\Gamma$. 


%

\begin{lemma}[{\cite[Lemma 4.1]{Bab78}}]\label{lm:babai} Let $(\Gamma,\cdot)$ be a group, $N$ a normal subgroup of $\Gamma$ and $C \subseteq \Gamma/N$. Then,
\[ \chi({\rm Cay}(\Gamma, C \cdot N)) \leq \chi({\rm Cay}(\Gamma/N, C)); \]
where $C \cdot N := \{c \cdot n \, \vert \, c \cdot N \in C,\, n \in N\}$. In particular, $\chi_{\mathrm{min}}(\Gamma)\leq \chi_{\mathrm{min}}(\Gamma/N)$.
\end{lemma}

Using this we get:
\begin{theorem}[Babai]\label{th:babai}
 For a finite group $\Gamma$ we have:
 $$\chi_{\mathrm{min}}(\Gamma)=\begin{cases}
                              1 & \text{if } \Gamma \text{ is trivial},\\
                              2 & \text{if } \Gamma \text{ has a subgroup of index }2,\\
                              3 & \text{otherwise.}\\
                             \end{cases}
$$
\end{theorem}
\begin{proof}
 The statement that $\chi_{\mathrm{min}}(\Gamma)=1$ if and only if $\Gamma$ is trivial, is trivial. Now, by Lemma~\ref{lm:babai} $\chi_{\mathrm{min}}(\Gamma)\leq \chi_{\mathrm{min}}(\Gamma/N)$ for any normal subgroup $N$ of $\Gamma$. If $\Gamma$ has a subgroup of index $2$, then it is normal and the quotient is $\mathbb{Z}_2$ and has a bipartite Cayley graph. In general, this argument reduces the question to simple groups. Since each such group $\Gamma'$ is either cyclic or has a generating set $C'$ of size $2$ containing an involution\cite{Kin17}, $\Cay(\Gamma', C')$ has maximum degree $3$ and is different from $K_4$. Hence $\chi(\Cay(\Gamma', C'))\leq 3$. It remains to show that if $\chi_{\mathrm{min}}(\Gamma)=2$, then $\Gamma$ has a subgroup of index $2$. In this case, $\Gamma$ may be partitioned into two independent sets $A,B$, which by vertex transitivity of $\Gamma$ are of equal size. If say $e\in A$, then $A$ consists of all elements of $\Gamma$ that can be expressed as a word of even length in $C$. Hence, $A$ is a subgroup of index $2$.
\end{proof}

We now turn to the \emph{maximum chromatic number} of a group $\Gamma$, i.e.,  $\chi_{\mathrm{max}}(\Gamma)$ the maximum $\chi(\Cay(\Gamma, C))$ over all minimal generating sets $C$ of $\Gamma$. We will show that this is at most $3$ for Dedekind groups, generalized dihedral groups, and nilpotent groups.

If $H<\Gamma$ is a subgroup and $C\subseteq \Gamma$, the \emph{Schreier (coset) graph} $\Cay(\Gamma/H,C)$ has as vertices the left cosets of $H$ and there is an edge between two cosets if they can be represented as  $gH$, $g'H$ and $g^{-1}g'\in C$.

\begin{lemma}\label{lm:Schreier}
 Let $C$ be a minimal generating set of $\Gamma$, then $$\chi(\Cay(\Gamma,C))\leq\max\{\chi(\Cay(\Gamma/\langle C-\{c\} \rangle,\{c\}))\mid c \in C\}.$$ 
\end{lemma}
\begin{proof}
Since $C=\{c_1,\ldots, c_k\}$ is minimal, for each $c\in C$ the graph $\Cay(\Gamma,C-\{c\})$ is disconnected and its connected components   correspond to the vertices of $\Cay(\Gamma/\langle C-\{c\}\rangle,\{c\})$. Moreover, if two vertices $x,y$ of $\Cay(\Gamma,C)$ are connected with an edge corresponding to $c$, then $x,y$ are contained in different components of $\Cay(\Gamma,C-\{c\})$ corresponding to adjacent vertices $x_c,y_c$ of $\Cay(\Gamma/\langle C-\{c\} \rangle,\{c\})$. Hence, if we map every vertex of $x\in \Cay(\Gamma,C)$ to the tuple of vertices $(x_{c_1}, \ldots, x_{c_k})$ we obtain a graph homomorphism into the Cartesian product $\Cay(\Gamma/ \langle C-\{c_1\}\rangle,\{c_1\})\square \cdots \square \Cay(\Gamma/\langle C-\{c_k\}\rangle,\{c_k\})$. Hence, the chromatic number of the latter is an upper bound for $\chi(\Cay(\Gamma,C))$. Finally, by a well-known result of Sabidussi~\cite{Sab57}, the chromatic number of a Cartesian product is the maximum chromatic number of its factors.
\end{proof}

\begin{remark}
 Note that Lemma~\ref{lm:Schreier} alone does not give a useful upper bound for general minimal Cayley graphs. Consider the so-called star graph~(see \cite{ADK94}), which is the bipartite graph $\Cay(S_n,C)$ of the symmetric group of degree $n$ with respect to the set of transpositions involving $1$, i.e., $C=\{(12),(13),\ldots, (1n)\})$. However, $\Cay(S_n/\langle C-\{c\} \rangle,\{c\}))=K_n$ for all $n$ and $c\in C$.
\end{remark}

We can however use the above lemma to bound the maximum chromatic number in some cases. A group is called \emph{Dedekind} if all its subgroups are normal. Clearly, this includes all abelian groups, and by results of Dedekind~\cite{Ded97} and Baer~\cite{Bae33}, actually not much more.

\begin{theorem}\label{thm:abelian3} Every minimal Cayley graph of a Dedekind group is $3$-colorable.
\end{theorem}
\begin{proof}
Consider ${\rm Cay}(\Gamma, C)$ a minimal Cayley graph of a Dedekind group $\Gamma$. Since $\Gamma$ is Dedekind, for all $c\in C$ we have that $\langle C-\{c\} \rangle$ is a normal subgroup of $\Gamma$. Hence, $\Gamma/\langle C-\{c\} \rangle$ is a cyclic group generated by the coset of $c$, i.e., the Schreier coset graph (which is a Cayley graph) is a cycle and $\chi({\rm Cay}(\Gamma/\langle C-\{c\} \rangle,\{c\})) \leq 3$. The statement follows from Lemma~\ref{lm:Schreier}.
\end{proof}

For any abelian group $\Gamma$, the generalized dihedral group of $\Gamma$, written ${\rm Dih}(\Gamma)$, is the semidirect product of $\Gamma$ and $\mathbb Z_2$, with $\mathbb Z_2$ acting on $\Gamma$ by inverting elements, i.e., 
${\rm Dih}(\Gamma)= \Gamma \rtimes_{\phi} \mathbb Z_2$ with $\phi(0)$ the identity, and $\phi(1)$ the inversion. Thus we get:
\begin{itemize}
\item[] $(g_1, 0) * (g_2, t_2) = (g_1 + g_2, t_2)$, and
\item[] $(g_1, 1) * (g_2, t_2) = (g_1 - g_2, 1 + t_2)$,
\end{itemize}
for all $g_1,g_2 \in \Gamma$, and $t_2 \in \mathbb Z_2$.

\begin{theorem}\label{thm:gendihedral3} Every minimal Cayley graph of a generalized dihedral group is $3$-colorable.
\end{theorem}
\begin{proof} Let $G = {\rm Cay}({\rm Dih}(\Gamma),C)$ be the Cayley graph of a generalized dihedral group ${\rm Dih}(\Gamma) = \Gamma \rtimes \mathbb Z_2$ minimally generated by  $C$. We denote $C_i = C \cap (\Gamma \times \{i\})$ for $i = 0,1$, and consider $H := (\Gamma \times \{0\}) \cap \langle C_1 \rangle$. By the minimality of $C$ we have that:
\begin{enumerate}
\item[(a)] the vertices corresponding to elements of $H$ form an independent set in $G$,
\item[(b)] The cosets of $C_0$ minimally generate the abelian group $(\Gamma \times \{0\}) / H$. 
\end{enumerate}
By (b) and Theorem \ref{thm:abelian3}, it follows that ${\rm Cay}( (\Gamma \times \{0\}) / H, C_0 )$ is $3$-colorable. We consider $\tilde{f}:(\Gamma \times \{0\}) / H \longrightarrow \{0,1,2\}$ a proper coloring. By (a), we have that ${f}: \Gamma \times \{0\} \longrightarrow \{0,1,2\}$ defined as $f(g,0) := \tilde{f}((g,0) * H)$ is also a proper coloring of ${\rm Cay}( \Gamma \times \{0\}, C_0 )$. Now we choose $(y,1) \in C_1$ and consider
\[ \begin{array}{rrcl} h: 	&  {\rm Dih}(\Gamma) & \longrightarrow & \{0,1,2\}  \\
							& (g,0)  & \mapsto & f(g,0) \\
							& (g,1) & \mapsto & f(g-y,0) + 1 \bmod 3. \end{array} \]
We claim that $h$ is a proper $3$-coloring of $G$. Indeed, consider $(g_1,t_1), (g_2,t_2)$ two adjacent vertices of $G$ and let us prove that $f(g_1,t_1) \neq f(g_2,t_2)$. We separate the proof in three cases:
\begin{enumerate}
\item If $t_1 = t_2 = 0$, then $(g_1,0)$ and $(g_2,0)$ are adjacent in ${\rm Cay}( \Gamma \times \{0\}, C_0 )$ and hence, $h(g_1,0) = f(g_1,0) \neq f(g_2,0) = h(g_2,0)$.
\item If $t_1 = t_2 = 1$, then $(g_1,1)$ and $(g_2,1)$ are adjacent if and only if  there exist $(x,0) \in C_0$ such that $(g_1,1) = (g_2,1) * (x,0) = (g_2 - x,1)$. Then, \[ (g_1 - y, 0) = (g_1,1) * (y,1) = (g_2-x,1) * (y,1) = (g_2 - y - x, 0)\] and, hence, $(g_1-y,0)$ and $(g_2 - y, 0)$ are adjacent in ${\rm Cay}( \Gamma \times \{0\}, C_0 )$. Thus, we conclude that $h(g_1,1) = f(g_1-y,0) + 1 \neq f(g_2 - y,0) + 1 = h(g_2,1)$.
\item If $t_1 = 0$, $t_2 = 1$, then $(g_1,0)$ and $(g_2,1)$ are adjacent if and only if  there exist $(z,1) \in C_1$ such that $(g_2,1) = (g_1,0) * (z,1) = (g_1+z,1)$. Then,  \[ h(g_2,1) =  f(g_1 + z - y,0) + 1 \bmod 3 = f(g_1,0) + 1 \bmod 3 \neq f(g_1,0) = h(g_1,0),\]
where the equality  $f(g_1 + z - y,0) = f(g_1,0)$ follows from the fact that $(z-y,0) = (z,1)*(y,1) \in H$. 
\end{enumerate}

This concludes the proof.
\end{proof}

For the next result, we denote by $\Phi(\Gamma)$ the \emph{Frattini subgroup} of $\Gamma$, that is, the intersection of all maximal proper subgroups of $\Gamma$, or $\Phi(\Gamma) = \{e\}$ if it has no maximal proper subgroups. 

\begin{lemma}\label{lm:Frattini}
Let $\Gamma$ be a group with Frattini subgroup $\Phi(\Gamma)$, then:
$$\chi_{\mathrm{max}}(\Gamma)\leq \chi_{\mathrm{max}}(\Gamma/\Phi(\Gamma)).$$
\end{lemma}
\begin{proof}
For $\Gamma$ a group and $C$ any minimal generating set. The following remarkable properties of the Frattini subgroup are well known  (see, e.g., \cite[Section 5.2]{Robinson}):
    \begin{enumerate}
        \item[{\rm (1)}] $\Phi(\Gamma)$ is a characteristic subgroup of $\Gamma$ and, hence, $\Phi(\Gamma) \unlhd \Gamma$,
        \item[{\rm (2)}] $\Phi(\Gamma)\cap C=\emptyset$, and
        \item[{\rm (3)}] $C/\Phi(\Gamma) = \{c \cdot \Phi(\Gamma) \, \vert \, c \in C\}$ is a minimal generating set of $\Gamma/\Phi(\Gamma)$.
    \end{enumerate}
  By (1), (2) and Lemma \ref{lm:babai}, one has that $\chi({\rm Cay}(\Gamma, C)) \leq \chi({\rm Cay}(\Gamma/\phi(\Gamma), C / \Phi(\Gamma)))$. By (3) $\chi({\rm Cay}(\Gamma/\phi(\Gamma), C / \Phi(\Gamma))) \leq \chi_{\rm max}(\Gamma/\phi(\Gamma)),$ and the result follows.
\end{proof}

For a group $(\Gamma,\cdot)$, we denote by $\Gamma'$ its commutator subgroup, i.e., \[\Gamma' = \langle x \cdot y \cdot x^{-1} \cdot y^{-1} \, \vert \, x,y \in \Gamma\rangle.\]

\begin{theorem}\label{thm:comfrattini} Let $\Gamma$ be a finitely generated group such  that its commutator $\Gamma'$ is contained in its Frattini subgroup $\Phi(\Gamma)$. Then, every minimal Cayley graph of $\Gamma$ is $3$-colorable. This includes nilpotent groups.
\end{theorem}
\begin{proof}Let $G = {\rm Cay}(\Gamma,C)$ be the Cayley graph of $\Gamma$ with respect to a minimal set of generators $C$.  
Since $\Gamma' \subseteq \Phi(\Gamma)$, then it follows that $\Gamma / \Phi(\Gamma)$ is commutative and minimally generated by $C / \Phi(\Gamma)$. Then, by Theorem \ref{thm:abelian3}, ${\rm Cay}(\Gamma / \Phi(\Gamma), C / \Phi(\Gamma))$ is $3$-colorable. But now we are done because since $C \subset (C / \Phi(\Gamma)) \cdot \Phi(\Gamma)$, by Lemma \ref{lm:babai} we have that ${\rm Cay}( \Gamma, C)$ is $3$-colorable.  
\end{proof}

A group $\Gamma$ satisfies that $\Gamma' < \Phi(\Gamma)$ if and only if all its maximal subgroups have prime index (see \cite[Theorem A]{Myro} for other characterizations of these groups). In particular, every nilpotent group satisfies that $\Gamma' \leq \Phi(\Gamma)$. A result of Wielandt \cite[5.2.16]{Robinson} proves that for a finite group $\Gamma$, one has that $\Gamma' \leq \Phi(\Gamma)$ if and only if $\Gamma$ is nilpotent. However, there are non-nilpotent infinite groups whose commutator is contained in its corresponding Frattini subgroup. A famous such group is the Grigorchuk group, which is a finitely generated $2$-group in which all maximal subgroups have index $2$. In fact $\Gamma' = \Phi(\Gamma)$, and has index $8$ in $\Gamma$~\cite{Gri80}. Other examples are considered in~\cite{FT21}.

Let us end this section with a general upper bound for the chromatic number of (semi)minimal Cayley graphs. For this purpose given a positive integer $n$, we denote by $W_b(n)$ the \emph{binary Lambert W function}, i.e., $n=W_b(n)2^{W_b(n)}$.
\begin{proposition}\label{pr:general}
  Let $\Gamma$ be a group of order $n$ and $C$ be a generating set. We have
  $$\chi(\Gamma, C)\leq \begin{cases}
                        2\log_2n & \text{if }C\text{ is semiminimal,}\\
                        2W_b(n) & \text{if }C\text{ is minimal.}
                       \end{cases}$$
Moreover, $W_b(n)<\log_2n-\log_2\log_2(\frac{n}{\log_2n})$.
\end{proposition}
\begin{proof}
 Let $C=(c_1,\ldots,c_k)$ a semiminimal generating set of $\Gamma$. One has that $\Gamma_{i-1}:= \langle c_1,\ldots,c_{i-1} \rangle$ is a proper subgroup of $\Gamma_{i}$ for all $1\leq i \leq k$ and, by Lagrange's Theorem, we have $2^i \leq |\Gamma_{i}|$. Hence $k=|C|\leq \log_2(|\Gamma|)=\log_2n$. Thus, the maximum degree of $\Cay(G,C)$ is at most $2\log_2n$. By Brook's Theorem, this is an upper bound for $\chi(\Cay(G,C))$ except if $\Cay(G,C)$ is an odd cycle or a clique. However, in the first case $\chi(\Cay(G,C))=3\leq 2\log_2(2\ell+1)$ for all $\ell\geq 1$. If otherwise $\Cay(G,C)=K_n$ is a clique, by Proposition~\ref{prop:clique} we know that $n=4$ and $\chi(\Cay(G,C))=4\leq 2\log_2(4)$.
 
 Now suppose $C$ minimal. Again by Lagrange's Theorem, for any $c\in C$ the subgroup $\langle C - \{c\} \rangle$ has order at least $2^{k-1}$. Hence, the Schreier graph of this subgroup has at most $\frac{n}{2^{k-1}}$ vertices, and $\chi(\Cay(\Gamma/\langle C - \{c\}\rangle,\{c\})) \leq \frac{n}{2^{k-1}}$. Thus, with the first part and Lemma~\ref{lm:Schreier} we have that $\chi(\Cay(\Gamma,C))\leq\min(2k,\frac{n}{2^{k-1}})$, which is maximised exactly if $k=W_b(n)$.
 By definition we have that $$W_b(n)=\log_2\left(\frac{n}{W_b(n)}\right)=\log_2\left(\frac{n}{\log_2\frac{n}{W_b(n)}}\right).$$ Moreover, 
 we clearly have clearly $W_b(n)<\log_2(n)$, then we finally get that  $$W_b(n) < \log_2\left(\frac{n}{\log_2\frac{n}{\log_2(n)}}\right)=\log_2n-\log_2\log_2\left(\frac{n}{\log_2n}\right),$$ which concludes the proof.
\end{proof}

Note that the bounds in the previous proposition depend on the maximal size of a minimal generating set of a group $\Gamma$. This parameter has been studied, see e.g.~\cite{LMS21,CC02}. 

\section{Lower bounds}
%

We begin this section by presenting a minimal Cayley graph of chromatic number $4$ and a semiminimal Cayley graph of chromatic number $7$.
 
 \begin{proposition}\label{prop:lb}
  There exist minimal Cayley graphs of chromatic number $4$ and semiminimal Cayley graphs of chromatic number $7$.
 \end{proposition}
 \begin{proof}
  For both graphs we computed the chromatic number by computer.
  The first graph is the minimal Cayley graph $\Cay(\mathbb{Z}_3\rtimes\mathbb{Z}_7,\{(0,1),(1,0)\})$. It is depicted on the left side of Figure~\ref{fig:examples} where the grey edges correspond to $(1,0)$ and the purple edges correspond to $(0,1)$. The second graph is the semiminimal Cayley graph $\Cay(Q_{32},C)$, where $Q_{32} = \langle a, b \, \vert \, a^{16} = b^4 = 1, a^8 = b^2, aba= b \rangle$ is the generalized quaternion group or order $32$ and $C = (b^2, a^4, a^{5}b,a^{3}b, a^{6}b)$. The graph $G$ depicted in the right of Figure~\ref{fig:examples} is such that $G\boxtimes K_2=\Cay(Q_{32},C)$, where $\boxtimes$ denotes the strong graph product. It is induced by vertices $\{b^i a^j \, \vert \, 0 \leq i \leq 1, 0 \leq j \leq 7\}$, edges corresponding to right multiplication by $a^4, a^5b, a^4b$ and $a^{14}b$ are depicted in blue, gray, green and violet, respectively. \end{proof}

\begin{figure}[h!]
\centering
\includegraphics[width=.8\textwidth]{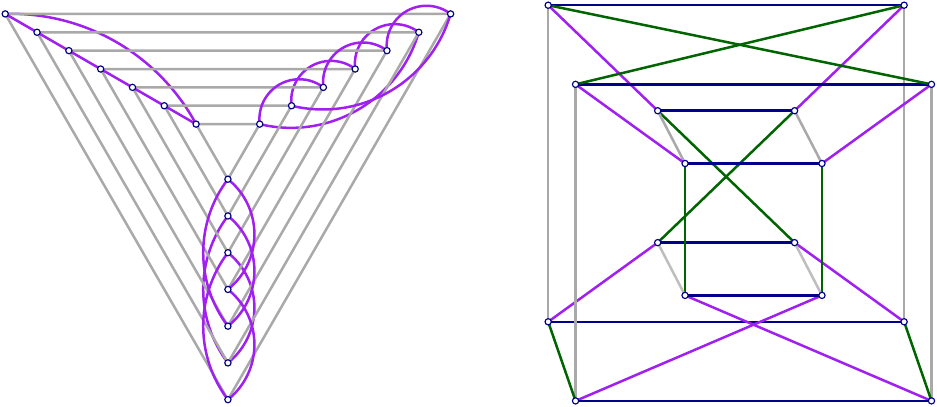}
\caption{\label{fig:examples} The graphs $\Cay(\mathbb{Z}_3\rtimes\mathbb{Z}_7,\{(0,1),(1,0)\})$ and $G$ such that $G\boxtimes K_2=\Cay(Q_{32},C)$ from Proposition~\ref{prop:lb}.}
\end{figure}

Already in~\cite{B78} it is shown that minimal Cayley graphs have clique number at most $3$ and it also follows from the results there, that semiminimal Cayley graphs have bounded clique number. We first make this precise.

\begin{proposition}\label{prop:clique}
 For any one popular color graph $G$ we have $\omega(G)\leq 5$ and this is tight. For a semiminimal Cayley graph $\Cay(\Gamma, C)$ we have $\omega(\Cay(\Gamma, C))\leq 4$ and this is tight.
\end{proposition}
\begin{proof}
 Suppose that there exists a one popular color coloring of $K_6$ and consider an edge $ab$ of color $1$ and the four triangles $abc,abd,abe,abf$.  At most two among these four triangles have popular color $1$. If two of these triangles, say $abc,abd$ have one popular color among different colors $2,3$, then the triangle $acd$ has no popular color. Hence two triangles say $abc,abd$ have  popular color $2$ then the triangles $abe,abf$ must have popular color $1$. So assume that the edge $ae$ is colored in $1$. But then the edges $ec$ and $ed$ both must be of color $1$ to make triangles $aec$ and $aed$ have a popular color. But then the degree of $e$ in color $1$ is $3$. 
 
 To show tightness just edge-color $K_5$ with two colors each inducing a cycle of length $5$. Since we are only using two  colors, it is straight-forward  to check that all cycles have a popular color. 
 Let us now, see that this is the only way to one popular color color the $K_5$. Let $ab$ an edge of color $1$ and consider the three triangles $abc,abd,abe$. If two of these triangles, say $abc,abd$ have one popular color because of different colors $2,3$, then the triangle $acd$ has no popular color. If two triangles say $abc,abd$ have one popular color because of color $2$ then the triangle $abe$ must have popular color $1$, say the edge $be$ is of color $1$ and both the edges $ce,de$ must be of color $1$ in order to make triangle  $bce,bde$ have a popular color. But then the degree of $e$ in color $1$ is $3$. Suppose now that only the triangle $abc$ has popular color $2$, then both $abd,abe$ have popular color $1$ and without loss of generality we have edges $ad,be$ of color $1$. Now $bd$ cannot be of color $1$  (because the degree of $b$ in color $1$ would be $3$), and cannot be of color $3$, because this would force the edge $cd$ to be of color $3$, but then the triangle $acd$ would have no popular color. Hence $bd$ is of color $2$, and by an analogous argument also $ae$ is of color $2$. Now, $cd$ is forced to be of color $1$, $de$ of color $2$, and $ce$ of color $1$. The resulting two-coloring is a decomposition into two cycles of length $5$.
 
 Suppose now that a semiminimal Cayley graph $\Cay(\Gamma, C)$ contains a $K_5$, then as argued above its one popular color coloring consists of two cycles of length $5$. Hence the corresponding elements $c,c'\in C$ generate the same cyclic subgroup of order $5$ of $\Gamma$. Hence $C$ is not semiminimal. Thus, $\omega(\Cay(\Gamma, C))\leq 4$. To see that this is tight simply consider $\Cay(\mathbb{Z}_4,\{2,1\}) = K_4$.
\end{proof}

We finish by giving a negative answer to the stronger version of Question \ref{conj:babai_old}, i.e., we provide a family of one popular color graphs with unbounded chromatic number. The construction we propose is based on one of the fundamental constructions for triangle-free graphs of arbitrary chromatic number due to Tutte (alias Blanche Descartes)~\cite{Des54}, also see Ne{\v{s}}et{\v{r}}il's survey~\cite{N13} for a more detailed discussion.

\begin{theorem} \label{th:countbabai}
 For any $k$ there exists a one popular color graph $G_k$ with $\chi(G_k)\geq k$.
\end{theorem}
 \begin{proof}
We proceed by induction on $k\geq 1$. Choose $G_1$ to be the graph with a single vertex and $G_2=C_4$ with edges colored alternatingly with two colors. If $k\geq 3$ then denote by $n$ the number of vertices of $G_{k-1}$. Define $X$ as a set of $(k-1)(n-1)+1$ new vertices without edges. Now, for every subset $Y$ of size $n$ of $X$ take a copy $G'$ of $G_{k-1}$ (where all copies can be considered to be edge-colored with the same set) and add a matching of size $n$ that connects the vertices of $Y$ with the vertices of $G'$. Each of these new matchings will be edge colored with its own private color. The resulting graph is $G_k$. Note that the choice of the matchings in the construction is not unique.

To see that $G_k$ is a one popular color graph, note first that every color class is a matching, hence the coloring satisfies property (1). 
Now, observe that by induction hypothesis all cycles within a single copy $G'$ have one color at least twice. Since $X$ is an independent set any other cycle must enter and leave some copy $G'$, but then it uses two edges of the same matching, hence repeats at least one color. This proves property (2').

The fact that $\chi(G_k)\geq k$ is well-known, see e.g.~\cite{Des54,N13}.
\end{proof}

While we have negatively answered the one popular color variant of Question~\ref{conj:babai_old} its no lonely color variant remains open. Let us propose a strengthening of it.
\begin{conjecture}
There is a function $f$, such that if the edges of a graph $G$ can be colored such the subgraph induced by any color has maximum degree $d$, and no color appears exactly once on a cycle of $G$, then $\chi(G)\leq f(d)$.
\end{conjecture}

Note that this conjecture for $d=2$ implies a positive answer to the weak variant of Question~\ref{conj:babai_old}. However, it is open even for the case $d=1$. On the other hand, it also remains open if graphs with a one popular color coloring exist that satisfy Conjecture~\ref{conj:babai_new}. Finally, note that the questions about the boundedness of the chromatic number of minimal or semiminimal Cayley graphs of general finite groups remain open.

\section{Acknowledgements}
We thank L\'aszl\'o Babai for several comments on an earlier version of this paper. 

The second author is partially supported by the Spanish Ministerio de Econom\'ia, Industria y Competitividad through grant RYC-2017-22701 and the Severo Ochoa and Mar\'ia de Maeztu Program for Centers and Units of Excellence in R\&D (CEX2020-001084-M).

Both authors are partially supported by the grant AcoGe (PID2022-137283NB-C22) funded by MCIN/AEI/10.13039/501100011033 and by ERDF/EU.

\bibliography{fraglit}
\bibliographystyle{my-siam}
\end{document}